\documentclass{amsart}
\usepackage{amsmath}
\usepackage{amssymb}
\usepackage{amsthm}
\usepackage{amsfonts}
\usepackage{graphicx}
\newtheorem{Theorem}{Theorem}[section]

\newtheorem{Definition}[Theorem]{Definition}
\newtheorem{Corollary}[Theorem]{Corollary}
\newtheorem{Lemma}[Theorem]{Lemma}
\newtheorem{Remark}[Theorem]{Remark}
\theoremstyle{definition}
\newtheorem{example}[Theorem]{Example}

\def \qed{\hfill{\hbox{$\square$}}}

\numberwithin{equation}{section}
\title{Timelike Loxodromes on Lorentzian Helicoidal Surfaces in Minkowski n--Space}

\author[B. Bekta\c s]{Burcu Bekta\c s Demirci}
\address{Fatih Sultan Mehmet Vak{\i}f University, Hali\c{c} Campus, Faculty of Engineering,
Department of Civil Engineering, 34445, Beyo\u{g}lu, İstanbul, Turkey, http://orcid.org/0000-0002-5611-5478.}
\email{bbektas@fsm.edu.tr(Corresponding author)}

\author[M. Babaarslan]{Murat Babaarslan}
\address{Yozgat Bozok University, Department of Mathematics, 66100, Yozgat, Turkey, http://orcid.org/0000-0012-2770-4126.}
\email{murat.babaarslan@bozok.edu.tr}

\author[Z. Öge]{Zehra \"{O}ge}
\address{Yozgat Bozok University, Graduate School of Natural and Applied Sciences, Department of Mathematics, 66100, Yozgat, Turkey, http://orcid.org/0000-0001-8184-4173}
\email{zehrahasimm66@gmail.com}

\begin{document}
\maketitle

\begin{abstract}
In this paper, we examine timelike loxodromes on three kinds of Lorentzian helicoidal surfaces 
in Minkowski $n$--space. 
First, we obtain the first order ordinary differential equations which determine timelike loxodromes 
on the Lorentzian helicoidal surfaces in $\mathbb{E}^n_1$ 
according to the causal characters of their meridian curves. 
Then, by finding general solutions, we get the explicit parametrizations of such timelike loxodromes. 
In particular, we investigate the timelike loxodromes on the three kinds of Lorentzian right
helicoidal surfaces in $\mathbb{E}^n_1$. Finally, we give an example to visualize the results. 
\end{abstract}

\section{Introduction}
Loxodromes, which are also known as rhumb lines, are curves that make 
constant angles with the meridians on the Earth's surface. 
Geodesics which minimize the distance between two points
on Earth's surface, are different from than loxodromes on Earth's surface, \cite{TW}.
Only the equator and the meridians are both constant course angle and length minimizing.  
Since loxodromes give an efficient routing from one position to 
another by means of a constant course angle, they are still primarily used 
in navigation. For details, we refer to \cite{WP, Alex, AP, TMJ}.  
Since the Earth's surface can be thought as a Riemannian sphere,  
the notion of loxodromes can be broaden to an arbitrary surface of revolution,
where meridians are copies of the profile curve.

In early of 20th century, C. A. Noble \cite{Noble} studied the loxodrome on the surface
of revolution in $\mathbb{E}^3$ and he also showed that the loxodrome
on sphereoid projects stereographically into the same spiral as the loxodrome on 
the sphere which is tangent to the sphereoid along equator. 
Then, S. Kos et al. \cite{KFH} and M. Petrovi\'{c} \cite{Petro} got the differential equations related to the loxodromes
on a sphere and a sphereoid and determined the length of such 
loxodromes, respectively.

Later, the topic of loxodromes have been studied on the rotational surfaces 
in Minkowski space which is important in general relativity. 
In 3--dimensional Minkowski space, there are three types of rotational surfaces 
with respect to the casual characters of rotation axes 
and the concept of angle to define loxodromes is not similar to Riemannian case. 
Therefore, the results in the Minkowski space are more richer than the Euclidean space. 
The authors determined the parametrizations of spacelike and timelike loxodromes 
on rotational surfaces in $\mathbb{E}^3_1$ which have either spacelike meridians or timelike meridians 
in \cite{BY1} and \cite{BM}, respectively.
For 4--dimensional Minkowski space, 
there are three types of rotation with 2--dimensional axes such as elliptic, hyperbolic and parabolic rotation
leaving a Riemannian plane, a Lorentzian plane or a degenerate plane pointwise fixed, respectively.  
Then, M. Babaarslan and M. G\"{u}m\"{u}\c{s} found the explicit parametrizations of loxodromes 
on such rotational surfaces of $\mathbb{E}^4_1$ in \cite{BG}.

Helicoidal surfaces are the natural generalizations of rotational surfaces 
and they play important roles in nature, science and engineering, see \cite{HR, J, PF}. 
Thus, this generalization leads the studies to the loxodromes on helicoidal surfaces in \cite{B1, BK2,BY2, BK1, B2}.
Recently, M. Babaarslan and N. S\"{o}nmez constructed the three kinds of 
helicoidal surfaces in $\mathbb{E}^4_1$ by using rotation with 2--dimensional axes and translation 
in $\mathbb{E}^4_1$ and they also obtained the general form of spacelike and timelike loxodromes 
on such helicoidal surfaces in \cite{BS}. 

With the motivation from geometry, M. Babaarslan, B. B. Demirci and R. Gen\c{c}
extended the notion of the helicoidal surfaces in $\mathbb{E}^4_1$ to 
higher dimensional Minkowski space and they made characterization of spacelike loxodromes 
on these helicoidal surfaces of $\mathbb{E}^n_1$ in \cite{BDG}.
In this context, this paper is a sequel of the article given by \cite{BDG}.   

In this paper, we study timelike loxodromes on three types of Lorentzian helicoidal surfaces 
in Minkowski $n$--space $\mathbb{E}^n_1$. 
We find the equations of timelike loxodromes on such helicodial surfaces which have either spacelike meridians or 
timelike meridians and then we get the explicit parametrizations of these loxodromes
by finding the general solution of the equations. 
As particular cases, we consider timelike loxodromes on each Lorentzian right
helicoidal surfaces in $\mathbb{E}^n_1$. Finally, we give an illustrative example.

\section{Preliminaries}


Let $\mathbb{E}^n_s$ denote the pseudo--Euclidean space of dimension $n$ and index $s$,
i.e.,\\
$\mathbb{R}^n=\{(x_1, x_2, \dots, x_n)\;|\; x_1, x_2, \dots, x_n \in\mathbb{R}\}$ equipped with the metric
\begin{equation}
ds^2=\sum_{i=1}^{n-s}dx_i^2-\sum_{j=n-s+1}^n dx_j^2.
\end{equation}
For $s=1$, $\mathbb{E}^n_1$ is known as the Minkowski space which is inspired by general relativity. 

A vector $v$ in $\mathbb{E}^n_1$ is called spacelike if $\langle v, v\rangle >0$ or $v=0$, 
timelike if $\langle v, v\rangle <0$, and lightlike (or null) if $\langle v, v\rangle=0$
and $v\neq 0$. The length of a vector $v$ in $\mathbb{E}^n_1$ 
is given by $||v||=\sqrt{|\langle v, v\rangle|}$ 
and $v$ is said to be an unit vector if $||v||=1$. 

Let $\alpha:I\subset\mathbb{R}\longrightarrow\mathbb{E}^n_1$ be a smooth regular curve
in $\mathbb{E}^n_1$, where $I$ is an open interval. 
Then, the causal character of $\alpha$ is spacelike, timelike or lightlike 
if $\alpha'$ is spacelike, timelike or lightlike, respectively, where $\alpha'=d\alpha/dt$.

Let $M$ be a pseudo--Riemannian surface in $\mathbb{E}^n_1$ given by a local parametrization ${\bf x}(u,v)$.
Then, the coefficients of the first fundamental form of $M$ are 
\begin{equation}
    E=\langle {\bf x}_u, {\bf x}_u \rangle,\;\;
    F=\langle {\bf x}_u, {\bf x}_v \rangle,\;\;
    G=\langle {\bf x}_v, {\bf x}_v \rangle
\end{equation}
where ${\bf x}_u$ and ${\bf x}_v$ denote the partial derivatives of ${\bf x}$ with respect $u$ and $v$, respectively.
Thus, the induced metric $g$ of $M$ in $\mathbb{E}^n_1$ is given by
\begin{equation}
    g=Edu^2+2Fdudv+Gdv^2.
\end{equation}
Also, a pseudo--Riemannian surface $M$ in $\mathbb{E}^n_1$ is called a spacelike surface or a timelike surface if
and only if $EG-F^2>0$ or $EG-F^2<0$, respectively. For the case $EG-F^2=0$, a pseudo--Riemannian surface $M$ is 
called a lightlike surface. Throughout this work, we will assume that the surface is nondegenerate.  
 
The length of the curve $\alpha$ on the pseudo--Riemannian surface $M$ between two points $u_0$ and $u_1$ in $\mathbb{E}^n_1$ is given by 
\begin{equation}
\label{length}
    L=\int_{u_0}^{u_1}\sqrt{\left |E+2F\frac{dv}{du}
    +G\left(\frac{dv}{du}\right)^2\right |}du.
\end{equation}
For later use, we give the following definition of Lorentzian angle in $\mathbb{E}^n_1$ by using \cite{Rat}. 
\begin{Definition}
Let $x$ and $y$ be vectors in $\mathbb{E}^n_1$. Then, we have the following statements:
\begin{itemize}
    \item [i.] for a spacelike vector $x$ and a timelike vector $y$, 
    there is a unique nonnnegative real number $\theta$ such that 
    \begin{equation}
    \label{ang3}
    \langle x,y\rangle=\pm ||x||||y||\sinh{\theta}.
    \end{equation}
    The number $\theta$ is called Lorentzian timelike angle 
    between $x$ and $y$. 
    
    \item [ii.] for timelike vectors $x$ and $y$,
    there is a unique nonnnegative real number $\theta$ such that 
    \begin{equation}
    \label{ang2}
    \langle x,y\rangle=||x||||y||\cosh{\theta}.
    \end{equation}
    The number $\theta$ is called Lorentzian timelike angle 
    between $x$ and $y$. 
    Note that $\theta=0$ if and only if $x$ and $y$ are positive scalar multiples of each other.
    \end{itemize}
\end{Definition}

By using \cite{BDG}, the definition of the helicoidal surfaces in $\mathbb{E}^n_1$ can be given as follows.

Let $\beta:I\subset\mathbb{R}\longrightarrow H\subset\mathbb{E}^n_1$ be a smooth curve in a hyperplane 
$\Pi\subset\mathbb{E}^n_1$, ${P}$ be a $(n-2)$-plane in the hyperplane $\Pi\subset\mathbb{E}^n_1$ and $\ell$ be a line parallel to ${P}$. A helicoidal surface in $\mathbb{E}^n_1$ is defined as a rotation of the curve $\beta$ around ${P}$ with a translation along the line $\ell$. Here, the speed of translation is proportional to the speed of this rotation. 
Thus, there are three types of helicoidal surfaces in $\mathbb{E}^n_1$ as follows:
\subsection{Helicoidal surface of type I}
Let $\{e_1,e_2,\dots, e_n\}$ be a standard orthonormal basis for $\mathbb{E}^n_1$.
Then, we choose a Lorentzian $(n-2)$--subspace ${\bf P}_I$ generated by 
$\{e_3, e_4,\dots, e_n\}$, ${\Pi}_I$ a hyperplane generated by 
$\{e_1, e_3,\dots,e_n\}$ and a line $\ell_I$ generated by $e_n$.
Assume that 
$\beta_I:I\longrightarrow\Pi_I\subset\mathbb{E}^n_1,
\;\beta_I(u)=(x_1(u), 0, x_3(u),\dots, x_n(u)),$
is a smooth regular curve lying in $\Pi_I$ defined on an open interval $I\subset\mathbb{R}$ and 
$u$ is arc length parameter, that is, $x_1'^2(u)+x_3'^2(u)+\cdots-x_n'^2(u)=\varepsilon$ with $\varepsilon=\pm 1$.
For $0\leq v<2\pi$ and a positive constant $c$, we consider the surface $M_I$
\begin{align}
\label{typeI}
H_1(u,v)=(x_1(u)\cos{v}, x_1(u)\sin{v}, x_3(u), \dots, x_{n-1}(u),  x_n(u)+cv)
\end{align}
which is the parametrization of the helicoidal surface obtained the rotation of the curve $\beta_1$ 
that leaves the Lorentzian subspace ${\bf P}_I$ pointwise fixed  
followed by the translation along $\ell_I$. The surface $M_I$ in $\mathbb{E}^n_1$
is called a helicoidal surface of type I.
Also, the surface $M_{I}$ is called a 
right helicoidal surface of type I in $\mathbb{E}^n_1$ if $x_n$ is a constant function.

\subsection{Helicoidal surface of type II}
Let $\{e_1,e_2,\dots, e_n\}$ be a standard orthonormal basis for $\mathbb{E}^n_1$.
Then, we choose a Riemannian $(n-2)$--subspace ${\bf P}_{II}$ generated by 
$\{e_1, e_2,\dots, e_{n-2}\}$,  
${\Pi}_{II}$ a hyperplane generated by 
$\{e_1, e_2,\dots,e_{n-2}, e_n\}$ and a line $\ell_{II}$ generated by $e_1$.
Assume that 
$\beta_{II}:I\longrightarrow\Pi_{II}\subset\mathbb{E}^n_1,
\;\beta_2(u)=(x_1(u), x_2(u),\dots,x_{n-2}(u),0, x_n(u)),$
is a smooth regular curve lying in $\Pi_{II}$ 
defined on an open interval $I\subset\mathbb{R}$ and $u$ is an arc length parameter, that is, 
$x_1'^2(u)+x_2'^2(u)+\cdots-x_n'^2(u)=\varepsilon$ for $\varepsilon=\pm 1$.
For $v\in\mathbb{R}$ and a positive constant $c$, we consider the surface $M_{II}$ 
\begin{align}
\label{typeII}
H_2(u,v)= (x_1(u)+cv, x_2(u), \dots, x_{n-2}(u), x_n(u)\sinh{v}, x_n(u)\cosh{v})
\end{align}
which is the parametrization of the helicoidal surface obtained the rotation of the curve $\beta_{II}$ which leaves 
Riemannian subspace ${\bf P}_{II}$ pointwise fixed followed by the translation along $\ell_{II}$. 
The surface $M_{II}$ in $\mathbb{E}^n_1$
is called a helicoidal surface of type II.
Also, the surface $M_{II}$ is called a 
right helicoidal surface of type II in $\mathbb{E}^n_1$ if $x_1$ is a constant function, .

\subsection{Helicoidal surface of type III}
Let define a pseudo--orthonormal basis 
$\{e_1, e_2, \cdots, \xi_{n-1}, \xi_n\}$ for $\mathbb{E}^n_1$
using a standard orthonormal basis $\{e_1,e_2,\cdots,e_{n-1},e_n\}$ for $\mathbb{E}^n_1$
such that 
\begin{equation}
    \xi_{n-1}=\frac{1}{\sqrt{2}}(e_n-e_{n-1})\;\;\mbox{and}\;\;
    \xi_{n}=\frac{1}{\sqrt{2}}(e_{n}+e_{n-1})
\end{equation}
where $\langle\xi_{n-1},\xi_{n-1}\rangle=\langle\xi_n,\xi_n\rangle=0$ and 
$\langle\xi_{n-1},\xi_n\rangle=-1$. 
Then, we choose a degenerate $(n-2)$--subspace ${\bf P}_{III}$ generated by $\{e_1, e_3, \cdots, \xi_{n-1}\}$, 
$\Pi_{III}$ a hyperplane generated by $\{e_1, e_3,\cdots,e_{n-2}, \xi_{n-1}, \xi_n \}$ and
a line $\ell_{III}$ generated by $\xi_{n-1}$. 
Assume that
$\beta_{III}:I\longrightarrow\Pi_{III}\subset\mathbb{E}^n_1,\;\beta_3(u)=x_1(u) e_1+ x_3(u)e_3+\cdots+x_{n-1}(u)\xi_{n-1}+x_n(u)\xi_n$ 
is a smooth curve lying in $\Pi_{III}$ defined on an open interval $I\subset\mathbb{R}$ 
and $u$ is an arc length parameter, that is, 
$x_1'^2(u)+x_3'^2(u)+\cdots-2x_{n-1}'(u)x_n'(u)=\varepsilon$ 
for $\varepsilon=\pm 1$.
Then, we consider the surface $M_{III}$ 
\begin{align}
\label{typeIII}
\begin{split}
H_3(u,v)=& x_1(u)e_1+\sqrt{2}vx_n(u)e_2+x_3(u)e_3+\cdots\\
 &+x_{n-2}(u)e_{n-2}
 +(x_{n-1}(u)+v^2x_n(u)+cv)\xi_{n-1}+x_n(u)\xi_n
 \end{split}
\end{align}
which is the parametrization of 
the helicoidal surface obtained a rotation 
of the curve $\beta_{III}$ which leaves the degenerate subspace ${\bf P}_{III}$ pointwise fixed
followed by the translation along $\ell_{III}$. The surface $M_{III}$ in $\mathbb{E}^n_1$
is called the helicoidal surface of type III.
If $x_n$ is a constant function, then the helicoidal surface $M_{III}$ is called a right helicoidal surface of type III 
in $\mathbb{E}^n_1$.
\begin{Remark}
It can be easily seen that the helicoidal surfaces $M_1$--$M_3$
in $\mathbb{E}^n_1$
defined by \eqref{typeI}, \eqref{typeII} and \eqref{typeIII} reduce to
the rotational surfaces in $\mathbb{E}^n_1$ for $c=0$. 
\end{Remark} 

\section{Timelike Loxodrome on Timelike Helicoidal Surface of Type I in $\mathbb{E}^n_1$}
In this section, we determine the parametrization of timelike loxodrome on the timelike helicoidal surface of type I in $\mathbb{E}^n_1$ defined by \eqref{typeI}. 

Consider the timelike helicoidal surface of type I, $M_1$, in $\mathbb{E}^n_1$ given by \eqref{typeI}. 
From a simple calculation, the induced metric $g_1$ on $M_1$ is defined by 
\begin{equation}
 g_1=\varepsilon du^2-2cx_n'(u)dudv+(x_1^2(u)-c^2)dv^2.  
\end{equation}
Since $M_1$ is a timelike surface in $\mathbb{E}^n_1$, 
we have $\varepsilon x_1^2(u)-c^2(\varepsilon+x_n'^2(u))<0$. 
Assume that $\alpha_1(t)=H_1(u(t),v(t))$ is a timelike loxodrome on $M_1$ in $\mathbb{E}^n_1$, that is,
$\alpha_1(t)$ intersects the meridian $m_1(u)=H_1(u,v_0)$ for a constant $v_0$ with a constant angle $\phi_0$ at the point $p\in M_1$. Then, we have 
\begin{align}
\label{type1eq1}
    &\tilde{g}(\dot{\alpha}(t),(m_I)_u)=
    \varepsilon\frac{du}{dt}-cx_n'(u)\frac{dv}{dt},\\
\label{type1eq2}    
    &\varepsilon\left(\frac{du}{dt}\right)^2-2cx_n'(u)\frac{du}{dt}\frac{dv}{dt}+(x_1^2(u)-c^2)\left(\frac{dv}{dt}\right)^2<0.
\end{align}
In this context, there are two following cases occur with respect to the causal character of the meridian curve $m_1(u)$.\\
\textit{Case i.} $M_1$ has a spacelike meridian curve $m_1(u)$,
that is, $\varepsilon=1$. Using the equations \eqref{type1eq1} and \eqref{type1eq2} in \eqref{ang3}, we get
\begin{equation}
\label{eq1}
    \sinh{\phi_0}=\pm\frac{\frac{du}{dt}-cx_n'(u)\frac{dv}{dt}}
    {\sqrt{-\left(\frac{du}{dt}\right)^2+2cx_n'(u)\frac{du}{dt}\frac{dv}{dt}-(x_1^2(u)-c^2)\left(\frac{dv}{dt}\right)^2}}.
\end{equation}
\\
\textit{Case ii.} $M_1$ has a timelike meridian curve $m_1(u)$, that is, $\varepsilon=-1$.  
Using the equations \eqref{type1eq1} and \eqref{type1eq2} in \eqref{ang2}, we obtain 
\begin{equation}
\label{eq2}
    \cosh{\phi_0}=-\frac{\frac{du}{dt}+cx_n'(u)\frac{dv}{dt}}
    {\sqrt{\left(\frac{du}{dt}\right)^2+2cx_n'(u)\frac{du}{dt}
    \frac{dv}{dt}-(x_1^2(u)-c^2)\left(\frac{dv}{dt}\right)^2}}.
\end{equation}
After a simple calculation in equations \eqref{eq1} and \eqref{eq2}, we get the following Lemma. 
\begin{Lemma}
\label{lemI}
Let $M_1$ be a timelike helicoidal surface of type I in $\mathbb{E}^n_1$ defined by \eqref{typeI}. Then,
$\alpha_1(t)=H_1(u(t),v(t))$ is a timelike loxodrome 
with $\dot{u}\neq 0$ if and only if one of the following differential equations is satisfied: 
\begin{itemize}
    \item [(i.)] for having a spacelike meridian,
    \begin{equation}
    \label{lem1eq1}
    (\sinh^2{\phi_0}(x_1^2(u)-c^2)+c^2x_n'^2(u))\dot{v}^2
    -2c\cosh^2\phi_0x_n'(u)\dot{u}\dot{v}+\cosh^2{\phi_0}\dot{u}^2=0
    \end{equation}
    
    \item [(ii.)] for having a timelike meridian,
    \begin{equation}
    \label{lem1eq2}
    (\cosh^2{\phi_0}(x_1^2(u)-c^2)+c^2x_n'^2(u))\dot{v}^2
    -2c\sinh^2\phi_0x_n'(u)\dot{u}\dot{v}-\sinh^2{\phi_0}\dot{u}^2=0
    \end{equation}
\end{itemize}
where $\phi_0$ is a nonnegative constant and $.$ denotes the derivative of any function respect to $t$.
\end{Lemma}

\begin{Theorem}
\label{thmI}
A timelike loxodrome on a timelike helicoidal surface of type I in $\mathbb{E}^n_1$ defined by \eqref{typeI} is parametrized by $\alpha_1(u)=H_1(u,v(u))$
where $v(u)$ is given by one of the following functions:
\begin{itemize}
    \item [(i.)] 
    $\displaystyle{v(u)=\pm\frac{1}{2\sinh{\phi_0}}\int_{u_0}^{u}
    \frac{d\xi}{\sqrt{c^2-x_1^2(\xi)}}}$, 
    
    \item [(ii.)] 
    $\displaystyle{v(u)=\pm\frac{1}{2\cosh{\phi_0}}\int_{u_0}^{u}
    \frac{d\xi}{\sqrt{c^2-x_1^2(\xi)}}}$,
    
    \item [(iii.)] for $\sinh^2\phi_0(x_1^2(\xi)-c^2)+c^2x_n'^2(\xi)\neq 0$,\\ 
    $\displaystyle{v(u)=\int_{u_0}^{u}
    \frac{2c\cosh^2{\phi_0}x_n'(\xi)
    \pm\sqrt{\sinh^2{(2\phi_0)}(c^2(x_n'^2(\xi)+1)-x_1^2(\xi))}}
    {2\sinh^2\phi_0(x_1^2(\xi)-c^2)+2c^2x_n'^2(\xi)}d\xi}$,
    
    \item [(iv.)] for  $\cosh^2\phi_0(x_1^2(\xi)-c^2)+c^2x_n'^2(\xi)\neq 0$,\\
    $\displaystyle{v(u)=\int_{u_0}^{u}
    \frac{2c\sinh^2{\phi_0}x_n'(\xi)
    \pm\sqrt{\sinh^2{(2\phi_0)}(c^2(x_n'^2(\xi)-1)+x_1^2(\xi))}}
    {2\cosh^2\phi_0(x_1^2(\xi)-c^2)+2c^2x_n'^2(\xi)}d\xi}$
\end{itemize}
where $\phi_0$ is a nonnegative constant and $c>0$ is a constant. 
\end{Theorem}

\begin{proof}
Assume that $M_1$ is a timelike helicoidal surface in $\mathbb{E}^n_1$ defined by \eqref{typeI} and $\alpha_1(t)=H_1(u(t),v(t))$ is a timelike 
loxodrome on $M_1$ in $\mathbb{E}^n_1$. From Lemma \ref{lemI},
we have the equations \eqref{lem1eq1} and \eqref{lem1eq2}. 

For a spacelike meridian, the equation \eqref{lem1eq1} implies  
\begin{equation}
\label{peq1}
     (\sinh^2{\phi_0}(x_1^2(u)-c^2)+c^2x_n'^2(u))
     \left(\frac{dv}{du}\right)^2
    -2c\cosh^2\phi_0x_n'(u)\frac{du}{dv}+\cosh^2{\phi_0}=0.
\end{equation}
If $\sinh^2{\phi_0}(x_1^2(u)-c^2)+c^2x_n'^2(u)=0$, 
then the equation \eqref{peq1} becomes
\begin{equation}
2c\cosh^2\phi_0x_n'(u)\frac{dv}{du}-\cosh^2{\phi_0}=0
\end{equation}
whose the solution is $v(u)=\frac{1}{2c}\int_{u_0}^u\frac{d\xi}{x_n'(\xi)}$. 
On the other side, $\sinh^2{\phi_0}(x_1^2(u)-c^2)+c^2x_n'^2(u)=0$ implies $x_n'(u)=\pm\frac{\sinh{\phi_0}}{c}\sqrt{c^2-x_1^2(u)}$ 
for $\phi_0\neq 0$, 
Thus, we get the desired equation in (i). Also, we note that 
$c^2-x_1^2(u)>0$ due the fact that $M_I$ is a timelike surface in $\mathbb{E}^n_1$. 

If $\sinh^2{\phi_0}(x_1^2(u)-c^2)+c^2x_n'^2(u)\neq 0$, 
it can be easily obtained that the solution $v(u)$ of the differential equation \eqref{peq1} is given by the integral in (iii).

Similarly, for a timelike meridian, the equation \eqref{lem1eq2} implies  
\begin{equation}
\label{peq2}
   (\cosh^2{\phi_0}(x_1^2(u)-c^2)+c^2x_n'^2(u))
   \left(\frac{dv}{du}\right)^2
    -2c\sinh^2\phi_0x_n'(u)\frac{dv}{du}-\sinh^2{\phi_0}=0
\end{equation}
If $\cosh^2{\phi_0}(x_1^2(u)-c^2)+c^2x_n'^2(u)=0$, 
the equation \eqref{peq2} reduces to the following equation
\begin{equation}
2c\sinh^2\phi_0x_n'(u)\frac{dv}{du}+\sinh^2{\phi_0}=0
\end{equation}
whose the solution is $v(u)=-\frac{1}{2c}\int_{u_0}^u\frac{d\xi}{x_n'(\xi)}$ for 
a nonzero constant $\phi_0$. 
Since $x_n'(u)=\pm\frac{\cosh{\phi_0}}{c}\sqrt{c^2-x_1^2(u)}$, 
we get the desired equation in (ii). Also, we note that 
$c^2-x_1^2(u)>0$ due the fact that $M_1$ is a timelike surface in $\mathbb{E}^n_1$. 
If  $\cosh^2{\phi_0}(x_1^2(u)-c^2)+c^2x_n'^2(u)\neq 0$, 
the solution $v(u)$ of the differential equation \eqref{peq2} is given by the integral in (iv). Thus, we get
the parametrization of the loxodrome with respect to $u$ parameter such that $\alpha_1(u)=H_1(u, v(u))$ where $v(u)$ is defined by one of the integrals in (i)-(iv). \qed
\end{proof}
Now, we consider a timelike right helicoidal surface of type I in $\mathbb{E}^n_1$, denoted by $M_{1}^{R}$, that is, 
\begin{equation}
\label{right1}
    H_{1}^{R}(u,v)=(x_1(u)\cos{v}, x_1(u)\sin{v}, x_3(u),\dots, x_{n-1}(u), x_{n_0}+cv)
\end{equation}
where $c\neq 0$ and $x_{n_0}$ are constants. 
Then, from the equations in (iii) and (iv) of Theorem \ref{thmI}, we give the following corollary. 
\begin{Corollary}
\label{CorI}
A timelike loxodrome on a timelike right helicoidal surface of type I in $\mathbb{E}^n_1$ defined by \eqref{right1} is parametrized by $\alpha_1^R(u)=H_{1}^{R}(u,v(u))$
where $v(u)$ is given by 
\begin{equation}
v(u)=\pm\coth{\phi_0}\int_{u_0}^u\frac{d\xi}
    {\sqrt{c^2-x_1^2(\xi)}}
\end{equation}   
for constant $\phi_0>0$. 
\end{Corollary}
Using the equation \eqref{length} and Corollary \ref{CorI}, we give the following statement:
\begin{Corollary}
The length of a timelike loxodrome on a timelike right helicoidal surface of type I in $\mathbb{E}^n_1$ defined by \eqref{right1} between two points $u_0$ and $u_1$ is given by 

$$L=\left|\frac{u_1-u_0}{\sinh{\phi_0}}\right|$$
    
where $\phi_0$ is a nonnegative constant. 
\end{Corollary}

\section{Timelike Loxodrome on Timelike Helicoidal Surface of Type II in $\mathbb{E}^n_1$}
In this section, we determine the parametrization of timelike loxodrome on the timelike helicoidal surface of type II in $\mathbb{E}^n_1$ defined by \eqref{typeII}. 

Consider the timelike helicoidal surface of type II, $M_2$, in $\mathbb{E}^n_1$ given by \eqref{typeII}. 
From a simple calculation, the induced metric $g_2$ on $M_2$ is defined by 
\begin{equation}
 g_2=\varepsilon du^2+2cx_1'(u)dudv+(c^2+x_n^2(u))dv^2.  
\end{equation}
Since $M_2$ is a timelike surface in $\mathbb{E}^n_1$, 
we have $c^2(\varepsilon-x_1'^2(u))+\varepsilon x_n^2(u)<0$. 
Assume that $\alpha_2(t)=H_2(u(t),v(t))$ is a timelike loxodrome on $M_2$ in $\mathbb{E}^n_1$, that is,
$\alpha_2(t)$ intersects the meridian $m_2(u)=H_2(u,v_0)$ for a constant $v_0$ with a constant angle $\phi_0$ at the point $p\in M_2$. 
Then, we have 
\begin{align}
\label{type2eq1}
    &\tilde{g}(\dot{\alpha}_2(t),(m_2)_u)=
    \varepsilon\frac{du}{dt}+cx_1'(u)\frac{dv}{dt},\\
\label{type2eq2}    
    &\varepsilon\left(\frac{du}{dt}\right)^2+2cx_1'(u)\frac{du}{dt}\frac{dv}{dt}+(c^2+x_n^2(u))\left(\frac{dv}{dt}\right)^2<0.
\end{align}
In this context, there are two following cases occur with respect to the causal character of the meridian curve $m_2(u)$.\\
\textit{Case i.} $M_2$ has a spacelike meridian curve $m_2(u)$,
that is, $\varepsilon=1$. Using the equations \eqref{type2eq1} and \eqref{type2eq2} in \eqref{ang3}, we get
\begin{equation}
\label{eq3}
    \sinh{\phi_0}=\pm\frac{\frac{du}{dt}+cx_1'(u)\frac{dv}{dt}}
    {\sqrt{-\left(\frac{du}{dt}\right)^2-2cx_1'(u)\frac{du}{dt}\frac{dv}{dt}-(c^2+x_n^2(u))\left(\frac{dv}{dt}\right)^2}}.
\end{equation}
\\
\textit{Case ii.} $M_2$ has a timelike meridian curve $m_2(u)$, that is, $\varepsilon=-1$. 
Using the equations \eqref{type2eq1} and \eqref{type2eq2} in \eqref{ang2}, we obtain 
\begin{equation}
\label{eq4}
    \cosh{\phi_0}=\frac{-\frac{du}{dt}+cx_1'(u)\frac{dv}{dt}}
    {\sqrt{\left(\frac{du}{dt}\right)^2-2cx_1'(u)\frac{du}{dt}
    \frac{dv}{dt}-(c^2+x_n^2(u))\left(\frac{dv}{dt}\right)^2}}.
\end{equation}
After a simple calculation in equations \eqref{eq3} and \eqref{eq4}, we get the following lemma. 
\begin{Lemma}
\label{lemII}
Let $M_2$ be a timelike helicoidal surface of type II in $\mathbb{E}^n_1$ defined by \eqref{typeII}. Then,
$\alpha_2(t)=H_2(u(t),v(t))$ is a timelike loxodrome 
with $\dot{u}\neq 0$ if and only if one of the following differential equations is satisfied: 
\begin{itemize}
    \item [(i.)] for having a spacelike meridian,
    \begin{equation}
    \label{lem2eq1}
    (\sinh^2{\phi_0}(x_n^2(u)+c^2)+c^2x_1'^2(u))\dot{v}^2
    +2c\cosh^2{\phi_0}x_1'(u)\dot{u}\dot{v}+\cosh^2{\phi_0}\dot{u}^2=0
    \end{equation}
    
    \item [(ii.)] for having a timelike meridian,
    \begin{equation}
    \label{lem2eq2}
    (\cosh^2{\phi_0}(x_n^2(u)+c^2)+c^2x_1'^2(u))\dot{v}^2
    +2c\sinh^2{\phi_0}x_1'(u)\dot{u}\dot{v}-\sinh^2{\phi_0}\dot{u}^2=0
    \end{equation}
\end{itemize}
 where $\phi_0$ is a nonzero constant and $.$ denotes the derivative of any function respect to $t$. 
\end{Lemma}

\begin{Theorem}
\label{thmII}
A timelike loxodrome on a timelike helicoidal surface of type II in 
$\mathbb{E}^n_1$ defined by \eqref{typeII} is parametrized by $\alpha_2(u)=H_2(u,v(u))$
where $v(u)$ is given by one of the following functions:
\begin{itemize}
    \item [(i.)] 
    $\displaystyle{v(u)=\int_{u_0}^{u}
    \frac{-2c\cosh^2{\phi_0}x_1'(\xi)
    \pm\sqrt{\sinh^2{2\phi_0}(c^2(x_1'^2(\xi)-1)-x_n^2(\xi))}}
    {2\sinh^2\phi_0(x_n^2(\xi)+c^2)+2c^2x_1'^2(\xi)}d\xi}$,
    
    \item [(ii.)] 
    $\displaystyle{v(u)=\int_{u_0}^{u}
    \frac{-2c\sinh^2{\phi_0}x_1'(\xi)
    \pm\sqrt{\sinh^2{2\phi_0}(x_n^2(\xi)+c^2(x_1'^2(\xi)+1))}}
    {2\cosh^2\phi_0(x_n^2(\xi)+c^2)+2c^2x_1'^2(\xi)}d\xi}$ 
\end{itemize}
where $\phi_0$ is a positive constant. 
\end{Theorem}

\begin{proof}
Assume that $M_2$ is a timelike helicoidal surface in $\mathbb{E}^n_1$ defined by \eqref{typeII} and $\alpha_2(t)=H_2(u(t),v(t))$ is a timelike 
loxodrome on $M_2$ in $\mathbb{E}^n_1$. From Lemma \ref{lemII},
we have the equations \eqref{lem2eq1} and \eqref{lem2eq2}. 

For a spacelike meridian, the equation \eqref{lem2eq1} implies  
\begin{equation}
\label{p2eq1}
 (\sinh^2{\phi_0}(x_n^2(u)+c^2)+c^2x_1'^2(u))
 \left(\frac{dv}{du}\right)^2
    +2c\cosh^2{\phi_0}x_1'(u)\frac{dv}{du}+\cosh^2{\phi_0}=0
\end{equation}
Since $\sinh^2{\phi_0}(x_n^2(u)+c^2)+c^2x_1'^2(u)\neq 0$ for all $u\in I\subset\mathbb{R}$, 
it can be easily obtained that the solution $v(u)$ of the differential equation \eqref{p2eq1} is given by the integral in (i).

Similarly, for a timelike meridian, the equation \eqref{lem2eq2} implies  
\begin{equation}
\label{p2eq2}
   (\cosh^2{\phi_0}(x_n^2(u)+c^2)+c^2x_1'^2(u))
   \left(\frac{dv}{du}\right)^2
   +2c\sinh^2{\phi_0}x_1'(u)\frac{dv}{du}-\sinh^2{\phi_0}=0
\end{equation}
Due to $\cosh^2{\phi_0}(x_n^2(u)+c^2)+c^2x_1'^2(u)\neq 0$, 
the solution $v(u)$ of the differential equation \eqref{p2eq2} is given by the integral in (ii). 
Thus, we get a parametrization of the loxodrome with respect to $u$ parameter such that 
$\alpha_2(u)=H_2(u, v(u))$ where $v(u)$ is defined by 
one of the integrals in (i) and (ii). \qed
\end{proof}

Now, we consider a timelike right helicoidal surface of type II in $\mathbb{E}^n_1$ denoted by $M_{2}^{R}$, that is, 
\begin{equation}
\label{right2}
    H_{2}^{R}(u,v)=(x_{1_0}+cv, x_2(u), \dots, x_{n-2}(u), x_n(u)\sinh{v}, x_n(u)\cosh{v})
\end{equation}
where $c\neq 0$ and $x_{1_0}$ are constants. 
Since $M_{2}^{R}$ is a timelike surface in $\mathbb{E}^n_1$, we have $\varepsilon(c^2+x_n^2(u))<0$.
This equality can be only satisfied when $\varepsilon=-1$. 
Thus, the meridian curve of $M_{2}^{R}$ must be timelike. 
Then, from the equations in (ii) of Theorem \ref{thmII}, we give the following corollary. 

\begin{Corollary}
\label{CorII}
A timelike loxodrome on a timelike right helicoidal surface of type II in $\mathbb{E}^n_1$ 
defined by \eqref{right2} is
parametrized by $\alpha_2^R(u)=H_{2}^{R}(u,v(u))$
where $v(u)$ is given by 
\begin{equation}
    v(u)=\pm\tanh{\phi_0}\int_{u_0}^u\frac{d\xi}
    {\sqrt{x_n^2(\xi)+c^2}}
\end{equation}
and $c, \phi_0>0$ are constants. 
\end{Corollary}

Using the equation \eqref{length} and Corollary \ref{CorII}, we give the following statement:
\begin{Corollary}
The length of a timelike loxodrome on a timelike right helicoidal surface of type II in $\mathbb{E}^n_1$ defined by \eqref{right2} between two points $u_0$ and $u_1$ is given by 
\begin{equation}
 \ell=\left|\frac{u_1-u_0}{\cosh{\phi_0}}\right|
\end{equation}
where $\phi_0$ is a nonnnegative constant.
\end{Corollary}

\section{Timelike Loxodrome on Timelike Helicoidal Surface of Type III in $\mathbb{E}^n_1$}
In this section, we determine the parametrization of timelike loxodrome on the timelike helicoidal surface of type III in $\mathbb{E}^n_1$ defined by \eqref{typeIII}. 

Consider the timelike helicoidal surface of type III, $M_3$, in $\mathbb{E}^n_1$ given by \eqref{typeIII}. 
The induced metric $g_3$ on $M_3$ is defined by 
\begin{equation}
 g_2=\varepsilon du^2-2cx_n'(u)dudv+2x_n^2(u)dv^2.  
\end{equation}
Since $M_3$ is a timelike surface in $\mathbb{E}^n_1$, 
we have $2\varepsilon x_n^2(u)-c^2x_n'^2(u)<0$. 
Assume that $\alpha_3(t)=H_3(u(t),v(t))$ is a timelike loxodrome on $M_3$ in $\mathbb{E}^n_1$, that is,
$\alpha_3(t)$ intersects the meridian $m_3(u)=H_3(u,v_0)$ for a constant $v_0$ with a constant angle $\phi_0$ at the point $p\in M_3$. 
Then, we have 
\begin{align}
\label{type3eq1}
    &\tilde{g}(\dot{\alpha}_3(t),(m_3)_u)=
    \varepsilon\frac{du}{dt}-cx_n'(u)\frac{dv}{dt},\\
\label{type3eq2}    
    &\varepsilon\left(\frac{du}{dt}\right)^2-2cx_n'(u)\frac{du}{dt}\frac{dv}{dt}+2x_n^2(u)\left(\frac{dv}{dt}\right)^2<0.
\end{align}
In this context, there are two following cases occur with respect to the causal character of the meridian curve $m_3(u)$.\\
\textit{Case i.} $M_3$ has a spacelike meridian curve $m_3(u)$, that is, $\varepsilon=1$.  
Using the equations \eqref{type3eq1} and \eqref{type3eq2} in \eqref{ang3}, we get
\begin{equation}
\label{eq5}
    \sinh{\phi_0}=\pm\frac{\frac{du}{dt}-cx_n'(u)\frac{dv}{dt}}
    {\sqrt{-\left(\frac{du}{dt}\right)^2+2cx_n'(u)\frac{du}{dt}\frac{dv}{dt}-2x_n^2(u)\left(\frac{dv}{dt}\right)^2}}.
\end{equation}
\\
\textit{Case ii.} $M_3$ has a timelike meridian curve $m_3(u)$, that is, $\varepsilon=-1$. 
Using the equations \eqref{type3eq1} and \eqref{type3eq2} in \eqref{ang2}, we obtain 
\begin{equation}
\label{eq6}
    \cosh{\phi_0}=-\frac{\frac{du}{dt}+cx_n'(u)\frac{dv}{dt}}
    {\sqrt{\left(\frac{du}{dt}\right)^2+2cx_n'(u)\frac{du}{dt}
    \frac{dv}{dt}-2x_n^2(u)\left(\frac{dv}{dt}\right)^2}}.
\end{equation}
After a simple calculation in the equations \eqref{eq5} and \eqref{eq6}, 
we get the following lemma. 
\begin{Lemma}
\label{lemIII}
Let $M_3$ be a timelike helicoidal surface of type III in $\mathbb{E}^n_1$ defined by \eqref{typeIII}. Then,
$\alpha_3(t)=H_3(u(t),v(t))$ is a timelike loxodrome 
with $\dot{u}\neq 0$ if and only if one of the following differential equations is satisfied:
\begin{itemize}
    \item [(i.)] for a spacelike meridian,
    \begin{equation}
    \label{lem3eq1}
    (2\sinh^2{\phi_0}x_n^2(u)+c^2x_n'^2(u))\dot{v}^2
    -2c\cosh^2{\phi_0}x_n'(u)\dot{u}\dot{v}+\cosh^2{\phi_0}\dot{u}^2=0
    \end{equation}
    
    \item [(ii.)] for a timelike meridian,
    \begin{equation}
    \label{lem3eq2}
    (2\cosh^2{\phi_0}x_n^2(u)+c^2x_n'^2(u))\dot{v}^2
    -2c\sinh^2{\phi_0}x_n'(u)\dot{u}\dot{v}-\sinh^2{\phi_0}\dot{u}^2=0
    \end{equation}
\end{itemize}
 where $\phi_0$ is a positive constant and $.$ 
 denotes the derivative of any function respect to $t$. 
\end{Lemma}

\begin{Theorem}
\label{thmIII}
A timelike loxodrome on a timelike helicoidal surface of type III in $\mathbb{E}^n_1$ defined by \eqref{typeIII} is parametrized by $\alpha_3(u)=H_3(u,v(u))$
where $v(u)$ is given by one of the following functions:
\begin{itemize}
    \item [(i.)] 
    $\displaystyle{v(u)=\int_{u_0}^{u}
    \frac{2c\cosh^2{\phi_0}x_n'(\xi)
    \pm\sqrt{\sinh^2{2\phi_0}(c^2x_n'^2(\xi)-2x_n^2(\xi)))}}
    {4\sinh^2\phi_0 x_n^2(\xi)+2c^2x_n'^2(\xi)}d\xi}$,
    
    \item [(ii.)] 
    $\displaystyle{v(u)=\int_{u_0}^{u}
    \frac{2c\sinh^2{\phi_0}x_n'(\xi)
    \pm\sqrt{\sinh^2{2\phi_0}(2x_n^2(\xi)+c^2x_n'^2(\xi))}}
    {4\cosh^2\phi_0 x_n^2(\xi)+2c^2x_n'^2(\xi)}d\xi}$ 
\end{itemize}
where $\phi_0$ is a positive constant. 
\end{Theorem}

\begin{proof}
Assume that $M_3$ is a timelike helicoidal surface in $\mathbb{E}^n_1$ defined by \eqref{typeIII} and $\alpha_3(t)=H_3(u(t),v(t))$ is a timelike 
loxodrome on $M_3$ in $\mathbb{E}^n_1$. From Lemma \ref{lemIII}, we have the equations \eqref{lem3eq1} and \eqref{lem3eq2}. 

For a spacelike meridian, the equation \eqref{lem3eq1} implies 
\begin{equation}
\label{p3eq1}
    (2\sinh^2{\phi_0}x_n^2(u)+c^2x_n'^2(u))
    \left(\frac{dv}{du}\right)^2-2c\cosh^2{\phi_0}x_n'(u)\frac{dv}{du}+\cosh^2{\phi_0}=0.
\end{equation}
Since $2\sinh^2{\phi_0}x_n^2(u)+c^2x_n'^2(u)\neq 0$ for all $u\in I\subset\mathbb{R}$, 
it can be easily obtained that the solution $v(u)$ of the differential equation \eqref{p3eq1} is given by the integral in (i).

Similarly, for a timelike meridian, the equation \eqref{lem3eq2} implies  
\begin{equation}
    \label{p3eq2}
    (2\cosh^2{\phi_0}x_n^2(u)+c^2x_n'^2(u))
    \left(\frac{dv}{du}\right)^2
    -2c\sinh^2{\phi_0}x_n'(u)\frac{dv}{du}-\sinh^2{\phi_0}=0.
\end{equation}
Due to $2\cosh^2{\phi_0}x_n^2(u)+c^2x_n'^2(u)\neq 0$, 
the solution $v(u)$ of the differential equation \eqref{p3eq2} is given by the integral in (ii). 
Thus, we get the parametrization of the loxodrome with respect to $u$ parameter such that 
$\alpha_3(u)=H_3(u, v(u))$ where $v(u)$ is defined by 
one of the integrals in (i) and (ii). \qed
\end{proof}

Note that the timelike right helicoidal surface with the timelike meridian does not exist.




\section{Visualization}
In this section, we give an example to visualize our main results.

\begin{example}

We consider the following spacelike profile curve:

\[\beta_I(u)=(x_1(u), 0, x_3(u),\dots, x_n(u)).\]

Then, we have the following parametrization of timelike helicoidal surface $M_{1}$:

\[H_1(u,v)=(x_1(u)\cos{v}, x_1(u)\sin{v}, x_3(u), \dots, x_{n-1}(u),  x_n(u)+cv).\]

By using Theorem 3.2 (i), we have  $v(u)=\pm\frac{1}{2c\sinh{\phi_0}}\int_{u_0}^{u}
    \frac{d\xi}{\sqrt{1-x_1^2(\xi)}}$.

If we choose $x_1(\xi)=k\sin{\xi}$ for $0<k<1$, then $v(u)=\pm\frac{1}{2c\sinh{\phi_0}}\int_{u_0}^{u}
    \frac{d\xi}{\sqrt{1-k^2\sin^2{\xi}}}=F(u,k)$ is an elliptic integral of first kind (see \cite{BF}).
    
Then, the parametrization of timelike loxodrome on timelike helicoidal surface $M_{1}$ in 
Minkowski $n$--space is given by

\[\alpha_1(u)=(x_1(u)\cos{v(u)}, x_1(u)\sin{v(u)}, x_3(u), \dots, x_{n-1}(u),  x_n(u)+cv(u)),\]

where $v(u)=F(u,k)$ for $0<k<1$.

\end{example}

\section{Conclusion}
Loxodromes on various surfaces and hypersurfaces in different ambient spaces
have been studied and many significant results have been obtained, see \cite{COP, Yoon, MTV, BDJ, C, Noble, BY1} etc.
In this paper, we investigate the timelike loxodromes on Lorentzian helicoidal surfaces in 
Minkowski $n$--space which were constructed in \cite{BDG}, called type I, type II and type III. 
For this reason, we get the first order ordinary differential equations which determine 
the parametrizations of the timelike loxodromes on such helicoidal surfaces. 
Solving these equations, we obtain the explicit parametrizations of the such loxodromes
parametrized by the parameter of the profile curves of the helicoidal surfaces. 
It is known that a particular case of helicoidal surfaces is right helicoidal surfaces.
We observe that the Lorentzian right helicoidal surfaces appear only for the 
Lorentzian helicoidal surfaces of type I having spacelike meridians and 
the Lorentzian helicoidal surfaces of type II having timelike meridians. 
Hence, we look the parametrizations for timelike loxodromes on which 
the Lorentzian right helicoidal of $\mathbb{E}^n_1$ exist. 
Moreover, we find the lengths of such loxodromes which just depend on the points and the angle 
between the loxodromes and the meridians of the surfaces.
Finally, we give a theoretical example to give the concept of the loxodromes.
In \cite{BS}, the graphical examples of the loxodromes can be found for the 4--dimensional Minkowski space.
Hence, our results in this paper and \cite{BDG} and can be used as finding the parametrizations of 
spacelike and timelike loxodromes on the nondegenerate helicoidal surfaces in the Minkowski space with the dimension higher 
than four.

\end{document}